\UseRawInputEncoding

\documentclass[12pt,draft]{article}
\usepackage{amsfonts}
\usepackage{amsmath}
\usepackage{amsthm}
\usepackage{amssymb}
\usepackage{mathrsfs}
\usepackage{newtxtext}

\newtheorem{dfz}{Definition}
\newtheorem{lem}{Lemma}
\newtheorem{teor}{Theorem}

\newtheorem{oss}{Remark}
\newcommand{\Ctond}{\mathaccent"017{C}}

\newcommand{\eps}{\varepsilon}
\newcommand{\vf}{\varphi}
\newcommand{\abs}[1]{\vert #1 \rvert}
\newcommand{\norma}[1]{\Vert #1 \Vert}

\title{Completeness theorems on the boundary for a parabolic equation}

\author{A.\ Cialdea
\thanks{Dipartimento di Scienze di Base e Applicate,
University of Ba\-si\-li\-ca\-ta, V.le dell'Ateneo Lucano, 10, 85100 Potenza, Italy.
 \textit{email:}
alberto.cialdea@unibas.it.}\and
C.S.\ Mare
\thanks{Dipartimento di Scienze di Base e Applicate,
University of Ba\-si\-li\-ca\-ta, V.le dell'Ateneo Lucano, 10, 85100 Potenza, Italy.
\textit{email:}
carmine.mare@unibas.it.}
}
\date{}    

\begin{document}

\maketitle

\begin{abstract}
Let $\{v_{\alpha}\}$ be a system of polynomial solutions of the parabolic equation $a_{hk}\partial_{x_{h}x_{k}}u - \partial_t u =0$ in a bounded $C^1$-cylinder $\Omega_{T}$ contained in $\mathbb{R}^{n+1}$. Here $a_{hk}\partial_{x_{h}x_{k}}$ is an elliptic operator with real constant coefficients. We prove that $\{v_{\alpha}\}$ is complete in $L^{p}(\Sigma')$, where $\Sigma'$ is the parabolic boundary of $\Omega_{T}$. Similar results are proved for the adjoint equation
$a_{hk}\partial_{x_{h}x_{k}} u+ \partial_t u =0$.
\end{abstract}

\textbf{Keywords:} Completeness theorems, Parabolic equations, Dirichlet problem, Partial differential equations with constant coefficients.

\textbf{MSC:} 42C30, 35K20, 35A35.

\section{Introduction}
The problem of the completeness of certain systems of  solutions - in particular
polynomial solutions - of a PDE on the
boundary has a long history. The prototype of these results is the completeness of harmonic polynomials on the boundary of a bounded domain. The  very first results proved in this field are due to \textsc{Fichera} 
\cite{Fichera1948}, who proved completeness theorems for harmonic polynomials related to
Dirichlet, Neumann, Robin and the mixed BVP for Laplace equation.  
For a full description of the problem and a review of several known results, we refer to Section 2 of \cite{C2019}. 
That paper contains a fairly complete list of references, to which we add
the more recent  \cite{CElasticNew} and \cite{CLanzCompl24}.

 Very few results have been obtained for parabolic equations in this field. Completeness theorems for the heat equation in a $C^2$-cylinder were proved by \textsc{Magenes} in his pioneeristic papers
\cite{MagenesNota1,MagenesNota2} (see also \textsc{Pagni} \cite{Pagni}).

Here we consider the more general parabolic operators
\begin{equation}
\label{eq:operatorecalore}
H=E-\frac{\partial}{\partial t}
\end{equation}
and its adjoint
\begin{equation}
\label{eq:operatorecalore*}
H^{*}=E+\frac{\partial}{\partial t}\, ,
\end{equation}
where $E$ is a second order elliptic operator with real constant coefficients and no lower order terms.

The aim of the present paper is to prove completeness theorems
related to the Dirichlet problem for the equation $Hu=0$ ($H^{*}u=0$)
in a $C^1$-cylinder $\Omega_{T}=\Omega\times [0,T]$.  More precisely, we shall prove that
the system of parabolic polynomials - i.e.\  polynomials satisfying the equation $Hu=0$ ($H^{*}u=0$)
- is complete in $L^{p}(\Sigma')$ ($1\leq p<\infty$), where
$\Sigma'$ is the parabolic boundary of $\Omega_{T}$ related to the operator $H$ ($H^{*}$).

The completeness results hinge on the study of certain classes of solutions,
which we denote by $\mathscr{A}^{p}(\Omega_{T})$ and  $\mathscr{A}^{p}_{*}(\Omega_{T})$. 
A key ingredient of this study is potential theory on $C^1$-cylinder. The relevant
fundamental results have been obtained by \textsc{Fabes} and \textsc{Rivi\`ere} 
\cite{FabesRivHeat}. They considered potentials for the heat equation, but their results
can be extended to the more general parabolic potentials  here considered. 

Another important ingredient is an approximation result due to \textsc{Malgrange} \cite{malgrange}.
This will enable us to deduce the completeness property of parabolic polynomials from a completeness property we obtain for certain potentials.

The paper is organized as follows. After some preliminaries given in Section \ref{sec:preliminary},
several properties of parabolic potentials are proved in Section \ref{sec:jump}.

Section \ref{sec:AeA*} is devoted to the study of the classes $\mathscr{A}^{p}(\Omega_{T})$ and  $\mathscr{A}^{p}_{*}(\Omega_{T})$. In particular some uniqueness theorems are obtained.

In Section \ref{sec:polynom} we construct a system of parabolic polynomials $\{v_{\alpha}\}$ $\left(\{w_{\alpha}\}\right)$,  such that
any polynomial satisfying the equation $Hu=0$ ($H^{*}u=0$) can be written as a finite linear combination of 
$\{v_{\alpha}\}$ $\left(\{w_{\alpha}\}\right)$.

Completeness theorems are then proved in Section \ref{sec:compl}.

\section{Preliminaries}\label{sec:preliminary}

Let $\Omega$ be a bounded domain of $\mathbb{R}^{n}$ ($n\geq 3$).
We assume always that $\mathbb{R}^{n}\setminus\overline{\Omega}$ is connected and
that $\partial\Omega$ is a $C^1$ boundary.

Given a positive $T$, we consider the $C^1$-cylinder $\Omega_{T}=\Omega\times (0,\,T)$ in $\mathbb{R}^{n+1}$. 
Points in $\mathbb{R}^{n+1}$ will be denoted by $(x,t)$, where $x\in\mathbb{R}^{n}$
and $t\in \mathbb{R}$, and
$\langle \cdot,\cdot\rangle$ stands for the inner product in $\mathbb{R}^{n}$.

Let $H$ be the parabolic operator \eqref{eq:operatorecalore}  and $H^{*}$ be its adjoint
\eqref{eq:operatorecalore*}, where
$$ 
Eu=\sum_{h,k=1}^{n} a_{hk}\frac{\partial^{2}u}{\partial x_h\,\partial x_k}\, ,
$$
$A=\{a_{hk}\}$ being a real symmetric and positive definite matrix.  $\abs{A}$ denotes the determinant of $A$ and by $A^{-1}$ we mean
the inverse matrix of $A$.

If $u,v\in C^{2}(\overline{\Omega}_T)$ the following Green's formula for the operator $H$ holds
\begin{equation}\label{eq:Green}
\int_{\Omega_T}[vH(u)-uH^{*}(v)]dx\,dt=-\int_{\Sigma_3}\biggl(v\frac{\partial u}{\partial\overline{\nu}}-u\frac{\partial v}{\partial\overline{\nu}}\biggr)d\sigma - \int_{\Sigma_1}uv\,d\sigma + \int_{\Sigma_2}uv\,d\sigma
\end{equation}
where 
\begin{equation*}
\Sigma_1=\{(x,T)\,\,|\,\,x\in\Omega\}\text{,}\,\,\Sigma_2=\{(x,0)\,\,|\,\,x\in\Omega\}\text{,}\,\,\Sigma_3=\partial\Omega\times (0,T),
\end{equation*}
 $\nu$ is the interior normal vector to $\Sigma=\partial\Omega_T$
and $\overline{\nu}$  is the conormal vector  defined by 
$(A\nu,0)$
(see, e.g., \cite[p.20]{OleinikRad}).

The fundamental solution for the parabolic equation $Hu=0$ is given by
$$
G(x,t)=
\begin{cases}
\frac{1}{(4\pi t)^{\frac{n}{2}}\abs{A}^{\frac{1}{2}}}\exp[-\frac{\langle A^{-1}x,x\rangle}{4t}], & \text{if}\ t>0,\\
0 ,  &\text{if}\ t\leq 0.\\
\end{cases}
$$

More precisely, $G(x-y,t-s)$ as a function of $(x,t)$ is of class $C^{\infty}(\mathbb{R}^{n+1}\setminus\lbrace(y,s)\rbrace)$ and is a fundamental solution for the operator $-H$ (see, e.g., \cite[p.146]{FollandPDEs}, where the result is given for the heat equation, but it can be easily adapted to $-H$), while as a function of $(y,s)$ is of class $C^{\infty}(\mathbb{R}^{n+1}\setminus\lbrace(x,t)\rbrace)$
and is a fundamental solution for the operator $-H^{*}$.

We observe that, by using the fundamental solution $G$, we have    a
``Stokes representation formula'' for  smooth enough functions:
\begin{equation}\label{eq:StokesH}
\begin{gathered}
-\int_{\Omega_{T}} G(x-y,t-s) \, Hu(y,s)\, dyds
\\+
\int_{\Sigma_3}\left(u(y,s)\, \frac{\partial }{\partial\overline{\nu}_{y,s}}G(x-y,t-s) - \frac{\partial u(y,s)}{\partial\overline{\nu}_{y,s}} \, G(x-y,t-s)\right)d\sigma_{y,s} \\
-\int_{\Sigma_1}u(y,T)\, G(x-y,t-T)\,dy +\int_{\Sigma_2}u(y,0)\, G(x-y,t)\, dy 
\\=
\begin{cases}
u(x,t)  & \text{if }  (x,t)\in\Omega_{T}, \\ 
0 & \text{if } (x,t) \notin \overline{\Omega}_{T}\, .
\end{cases}
\end{gathered}
\end{equation}
Note that, if $t<T$, in particular when $(x,t)\in\Omega_{T}$,
$$
\int_{\Sigma_1}u(y,T)\, G(x-y,t-T)\,dy=0.
$$

We have also
\begin{equation}\label{eq:StokesH*}
\begin{gathered}
-\int_{\Omega_{T}} G(x-y,t-s) \, H^{*}u(x,t)\, dxdt
\\
+\int_{\Sigma_3}\left(u(x,t)\, \frac{\partial }{\partial\overline{\nu}_{x,t}}G(x-y,t-s) - \frac{\partial u(x,t)}{\partial\overline{\nu}_{x,t}}\, G(x-y,t-s)\right)d\sigma_{x,t} \\
+\int_{\Sigma_1}u(x,T)\, G(x-y,T-s)\, dx -\int_{\Sigma_2}u(x,0)\, G(x-y,-s)\, dx 
\\=
\begin{cases}
u(y,s)  & \text{if }  (y,s)\in\Omega_{T}, \\ 
0 & \text{if } (y,s) \notin \overline{\Omega}_{T},
\end{cases}
\end{gathered}
\end{equation}
and
$$
\int_{\Sigma_2}u(x,0)\, G(x-y,-s)\, dx=0
$$
if $s>0$, in particular when $(y,s)\in\Omega_{T}$.

As a consequence of these representation formulas,  we have that, if 
$A$ is a domain in $\mathbb{R}^{n+1}$, then
\begin{equation}\label{eq:suppcomp}
\begin{gathered}
v(x,t)=-\int_{A}G(x-y,t-s) \, Hv(y,s)\, dyds, \quad \forall\ (x,t)\in \mathbb{R}^{n+1},\\
v(y,s)=-\int_{A}G(x-y,t-s) \, H^{*}v(x,t)\, dxdt, \quad \forall\ (y,s)\in \mathbb{R}^{n+1}
\end{gathered}
\end{equation}
for any $v\in \Ctond^{\infty}(A)$,  where $\Ctond^{\infty}(A)$ is the space of
$C^{\infty}$ functions with compact support contained in $A$.

As usual, we denote by $L^{p}(\Sigma)$ ($1\leq p<\infty$) 
 the vector space of all measurable real valued functions $f$
such that $\abs{f}^{p}$ is integrable over $\Sigma$.

\section{The jump formulas for  parabolic potentials}\label{sec:jump}

Let $\vf\in L^{p}(\Sigma_3)$ ($1\leq p<\infty$). The   double layer parabolic potential with density $\vf$ is defined as
\begin{equation*}
\int_{\Sigma_3}\vf(y,s)\,\frac{\partial G(x-y,t-s)}{\partial\overline{\nu}_{y,s}}\,d\sigma_{y,s}\, .
\end{equation*}

The next results show that the jump relations holding on the lateral boundary $\Sigma_{3}$ for caloric potentials also hold for this more general type of potentials.

\begin{teor}
\label{teor:salto_doppio}
Let $\vf\in L^{p}(\Sigma_3)$ ($1\leq p<\infty$). Then 
\begin{equation}
\label{eq:formulasalto_doppio}
\begin{gathered}
\lim_{(x,t)\to (x_{0},t)^{\pm}}\int_{\Sigma_3}\vf(y,s)\,\frac{\partial G(x-y,t-s)}{\partial\overline{\nu}_{y,s}}\,d\sigma_{y,s} =\pm\frac{1}{2}\vf(x_0,t)+\\
                                                         +\int_{\Sigma_3}\vf(y,s)\,\frac{\partial G(x_0-y,t-s)}{\partial\overline{\nu}_{y,s}}\,d\sigma_{y,s}
\end{gathered}
\end{equation} 
for almost every $(x_0,t) \in \Sigma_3$. Here the limit $(x,t)\to (x_0,t)^{+}$ ($(x,t)\to (x_0,t)^{-}$) has to be understood as 
an internal (external) angular boundary value and the integral on the right hand side exists as a singular integral. Moreover, for any $1<p<\infty$ and for any real $c\neq 0$, the operator
\begin{equation}\label{eq:coperator}
c\, \vf(x,t) + \int_{\Sigma_3}\vf(y,s)\,\frac{\partial G(x-y,t-s)}{\partial\overline{\nu}_{y,s}}\,d\sigma_{y,s}
\end{equation}
is invertible on $L^{p}(\Sigma_{3})$.
\end{teor}
\begin{proof}
We begin by considering the potential
$$
\int_{\Sigma_{3}^{\infty}}\,\frac{\partial G(x-y,t-s)}{\partial\overline{\nu}_{y,s}}\,d\sigma_{y,s} \, ,
$$
where  $\Sigma_{3}^{\infty}=\partial\Omega\times(-\infty,+\infty)$.  This is equal to
\begin{equation}\label{eq:equalto}
\frac{1}{2(4\pi)^{n/2}|A|^{1/2}} \int_{-\infty}^{t} \frac{ds}{(t-s)^{1+n/2}}
 \int_{\partial\Omega}
\langle x-y,\nu(y)\rangle   e^{-\frac{\langle A^{-1}(x-y),(x-y)\rangle}{4(t-s)}}\, d\sigma_{y}\, .
\end{equation}

Setting $u=\langle A^{-1}(x-y),(x-y)\rangle/(4(t-s))$, we find
\begin{gather*}
\frac{1}{2 \pi^{n/2}|A|^{1/2}}
\int_{0}^{+\infty}u^{\frac{n}{2}-1}e^{-u}du \int_{\partial\Omega}
\frac{\langle x-y,\nu(y)\rangle}{\langle A^{-1}(x-y),(x-y)\rangle^{n/2}} d\sigma_{y}
\\
=\frac{\Gamma(n/2)}{2 \pi^{n/2}|A|^{1/2}}
\int_{\partial\Omega}
\frac{\langle x-y,\nu(y)\rangle}{\langle A^{-1}(x-y),(x-y)\rangle^{n/2}} d\sigma_{y}\\
= \frac{1}{\omega_n |A|^{1/2}}\int_{\partial\Omega}
\frac{\langle x-y,\nu(y)\rangle}{\langle A^{-1}(x-y),(x-y)\rangle^{n/2}} d\sigma_{y}. 
\end{gather*}
This shows that, denoting by $s(x,y)$ the fundamental solution of the elliptic equation $Eu=0$, i.e.
$$
s(x,y)=\frac{1}{(2-n)\omega_{n}|A|^{1/2}}\langle A^{-1}(x-y),(x-y)\rangle^{(2-n)/2},
$$
we have
$$
\int_{\Sigma_{3}^{\infty}}\,\frac{\partial G(x-y,t-s)}{\partial\overline{\nu}_{y,s}}\,d\sigma_{y,s} =
- \int_{\partial\Omega} \frac{\partial}{\partial\overline{\nu}_{y}} s(x,y)\, d\sigma_y
$$
and then
\begin{equation}\label{eq:doppio1}
\int_{\Sigma_{3}^{\infty}}\,\frac{\partial G(x-y,t-s)}{\partial\overline{\nu}_{y,s}}\,d\sigma_{y,s} 
=
\begin{cases}
1  &(x,t)\in \Omega\times(-\infty,+\infty) \\ 
 \frac{1}{2} &(x,t)\in \partial\Omega\times(-\infty,+\infty)\\
 0 &  (x,t)\in (\mathbb{R}^n\setminus\overline{\Omega})\times(-\infty,+\infty).
\end{cases}
\end{equation}

Suppose now $p>1$ and let $\vf\in L^{p}(\Sigma_3)$.  
By means of the arguments used in  \cite[Th. 1.1]{FabesRivHeat} one can show that the integral 
\begin{equation*}
\int_{\Sigma_3}\vf(y,s)\frac{\partial G(x-y,t-s)}{\partial\overline{\nu}_{y,s}}\,d\sigma_{y,s}
\end{equation*}
does exist as a singular integral  for almost every $(x,t)\in\Sigma_3$.

In order to prove \eqref{eq:formulasalto_doppio}, consider first a function $\vf\in\Ctond^{1}(\Sigma_3)$. This means that
$\vf\in C^{1}(\Sigma^{\infty}_3)$ and its support is contained in $\partial\Omega\times(0,T)$. Let $(x_0,t)$ be in $\Sigma_3$; keeping in mind 
\eqref{eq:doppio1}, we have
\begin{gather*}
\int_{\Sigma_3}\vf(y,s)\frac{\partial G(x-y,t-s)}{\partial\overline{\nu}_{y,s}}d\,\sigma_{y,s} \\
=\int_{\Sigma^{\infty}_3}[\vf(y,s)-\vf(x_0,t)]\frac{\partial G(x-y,t-s)}{\partial\overline{\nu}_{y,s}}d\,\sigma_{y,s}
       + \vf(x_0,t)
\end{gather*}
for any $(x,t)\in \Omega_T$.  As far as the integral on the right hand side is concerned, we can write it as
\begin{equation*}
C
 \int_{-\infty}^{t}\frac{ds}{(t-s)^{1+n/2}} 
\int_{\partial\Omega}
\langle x-y,\nu(y)\rangle [\vf(y,s)-\vf(x_0,t)] e^{-\frac{\langle A^{-1}(x-y),(x-y)\rangle}{4(t-s)}}\, 
 d\sigma_{y}\, .
\end{equation*}
Similarly to what we did for \eqref{eq:equalto}, this integral is equal to
\begin{equation}\label{eq:integrand}
C
\int_{0}^{+\infty} u^{\frac{n}{2}-1}e^{-u}\, du 
\int_{\partial\Omega}\frac{\langle x-y,\nu(y)\rangle}{\langle A^{-1}(x-y),(x-y)\rangle^{\frac{n}{2}}} \Lambda(x_0,x,y,t,u)\, d\sigma_y
\end{equation}
where
$$
\Lambda(x_0,x,y,t,u)=\vf\left(x,t-\langle A^{-1}(x-y),(x-y)\rangle/4u\right)-\vf(x_0,t)\, .
$$

Since there exists a constant $K$ such that the modulus of the integrand in \eqref{eq:integrand} can be majorized by
$$
K\, \frac{u^{\frac{n}{2}-1}e^{-u}}{|y-x_0|^{n-2}},
$$
we may apply Lebesgue's dominated convergence theorem and, recalling \eqref{eq:doppio1}, conclude that 
\begin{gather*}
\lim_{(x,t)\to (x_0,t)^{+}}\int_{\Sigma_3}\vf(y,s)\,\frac{\partial G(x-y,t-s)}{\partial\overline{\nu}_{y,s}}\,d\sigma_{y,s}\\
 = \int_{\Sigma^{\infty}_3}[\vf(y,s)-\vf(x_0,t)]\frac{\partial G(x_0-y,t-s)}{\partial\overline{\nu}_{y,s}}d\,\sigma_{y,s}
       + \vf(x_0,t) \\
       =  \int_{\Sigma_3}\vf(y,s)\frac{\partial G(x_0-y,t-s)}{\partial\overline{\nu}_{y,s}}d\,\sigma_{y,s}
       - \vf(x_0,t)  \int_{\Sigma^{\infty}_3}\frac{\partial G(x_0-y,t-s)}{\partial\overline{\nu}_{y,s}}d\,\sigma_{y,s} \\
       + \vf(x_0,t) =
       \int_{\Sigma_3}\vf(y,s)\frac{\partial G(x_0-y,t-s)}{\partial\overline{\nu}_{y,s}}d\,\sigma_{y,s}+
       \frac{1}{2}\vf(x_0,t)\, .
\end{gather*}

The same proof works for $(x,t)\to (x_0,t)^{-}$. Formula \eqref{eq:formulasalto_doppio} is then proved if 
$\vf\in\Ctond^{1}(\Sigma_3)$.

Let now $\vf\in L^{p}(\Sigma_3)$.
Given $\alpha$, $0<\alpha<1$, consider the non tangential maximal function
\begin{equation*}
U^{*}(x,t)=\sup\{\abs{U(\xi,t)}\,\,:\,\,\xi\in\Omega\text{,}\,\,\abs{\xi-x}<\delta\text{,}\,\,\langle\xi-x,\nu(x)\rangle\,\,>\alpha\abs{\xi-x}\}\, ,
\end{equation*}
where
\begin{equation*}
U(x,t)=\int_{\Sigma_3}\vf(y,s)\frac{\partial G(x-y,t-s)}{\partial\overline{\nu}_{y,s}}\,d\sigma_{y,s}.
\end{equation*}
Reasoning as  in \cite[Th. 2.1]{FabesRivHeat}, there exists a constant $\delta=\delta_{\Omega,\alpha}$ such that  $U^*$
belongs to $L^{p}(\Sigma_3)$ and 
\begin{equation*}
\norma{U^{*}}_{L^{p}(\Sigma_3)}\leq C\norma{\vf}_{L^{p}(\Sigma_3)}
\end{equation*}
with $C$ independent of $\vf$. 
By standard arguments, the pointwise limit \eqref{eq:formulasalto_doppio} 
follows from what we have already proved when $\vf\in\Ctond^{1}(\Sigma_3)$
(see \cite[p.172--173]{FabesJodeitRiv}).

The case $p=1$  is slightly different, because the double layer potential is bounded
from $L^1(\Sigma_3)$ into weak-$L^1(\Sigma_3)$. Nevertheless the pointwise 
limit  \eqref{eq:formulasalto_doppio} is  still valid, see the Remark  in \cite[p.173]{FabesJodeitRiv}.

The invertibility of operator \eqref{eq:coperator} can be proved as in \cite[Th. 1.3]{FabesRivHeat}.
\end{proof}

A similar result holds for the conormal derivative of 
the single layer parabolic potential 
\begin{equation*}
\int_{\Sigma_3}\vf(y,s)G(x-y,t-s)\,d{\sigma}_{y,s} \, .
\end{equation*}

\begin{teor}
\label{teor:salto_conorm}
Let $\vf\in L^{p}(\Sigma_3)$ ($1\leq p<\infty$). Then 
\begin{equation}
\label{eq:formulasalto_conorm}
\begin{gathered}
\lim_{(x,t)\to (x_0,t)^{\pm}}\int_{\Sigma_3}\vf(y,s)\,\frac{\partial G(x-y,t-s)}{\partial\overline{\nu}_{x_{0},t}}\,d\sigma_{y,s} =\mp\frac{1}{2}\vf(x_0,t)+\\
                                                         +\int_{\Sigma_3}\vf(y,s)\,\frac{\partial G(x_0-y,t-s)}{\partial\overline{\nu}_{x_{0},t}}\,d\sigma_{y,s}
\end{gathered}
\end{equation} 
for almost every $(x_0,t) \in \Sigma_3$. As in Theorem \ref{teor:salto_doppio}, the limit $(x,t)\to (x_0,t)^{+}$ ($(x,t)\to (x_0,t)^{-}$) has to be understood as 
an internal (external) angular boundary value and the integral on the right hand side exists as a singular integral. 
Moreover, for any $1<p<\infty$ and for any real $c\neq 0$, the operator
\begin{equation}\label{eq:coperator2}
c\, \vf(x,t) + \int_{\Sigma_3}\vf(y,s)\,\frac{\partial G(x-y,t-s)}{\partial\overline{\nu}_{x,t}}\,d\sigma_{y,s}
\end{equation}
is invertible on $L^{p}(\Sigma_{3})$.
\end{teor}
\begin{proof}
The proof is similar to that of Theorem \ref{teor:salto_doppio}. The only difference concerns
the proof of the jump relation when  $\vf\in\Ctond^{1}(\Sigma_3)$. Now we have
$$
\int_{\Sigma_{3}^{\infty}}\,\frac{\partial G(x-y,t-s)}{\partial\overline{\nu}_{x,t}}\,d\sigma_{y,s} =
- \int_{\partial\Omega} \frac{\partial}{\partial\overline{\nu}_{x}} s(x,y)\, d\sigma_y
$$
and then a formula like \eqref{eq:doppio1} 
does not hold. However it is known  that
$$
 \lim_{x\to x_0^{\pm}} \int_{\partial\Omega} \frac{\partial}{\partial\overline{\nu}_{x_0}} s(x,y)\, d\sigma_y =
 \pm \frac{1}{2} + \int_{\partial\Omega} \frac{\partial}{\partial\overline{\nu}_{x_0}} s(x_0,y)\, d\sigma_y
$$
(see e.g. \cite[p.35]{Miranda})
and then
\begin{gather*}
\lim_{(x,t)\to (x_0,t)^{+}}\int_{\Sigma_3}\vf(y,s)\,\frac{\partial G(x-y,t-s)}{\partial\overline{\nu}_{x_{0},t}}\,d\sigma_{y,s}\\
= \lim_{(x,t)\to (x_0,t)^{+}}\Bigg(
\int_{\Sigma^{\infty}_3}[\vf(y,s)-\vf(x_0,t)]\frac{\partial G(x-y,t-s)}{\partial\overline{\nu}_{x_{0},t}}d\,\sigma_{y,s}\\
       + \vf(x_0,t) \int_{\Sigma^{\infty}_3}\frac{\partial G(x-y,t-s)}{\partial\overline{\nu}_{x_{0},t}}d\,\sigma_{y,s}
\Bigg)\\
 = \int_{\Sigma^{\infty}_3}[\vf(y,s)-\vf(x_0,t)]\frac{\partial G(x_0-y,t-s)}{\partial\overline{\nu}_{x_{0},t}}d\,\sigma_{y,s} \\
       + \vf(x_0,t) \Bigg( -\frac{1}{2} +  \int_{\Sigma^{\infty}_3}\frac{\partial G(x_{0}-y,t-s)}{\partial\overline{\nu}_{x_{0},t}}d\,\sigma_{y,s}\Bigg) \\
      = \int_{\Sigma_3}\vf(y,s)\frac{\partial G(x_0-y,t-s)}{\partial\overline{\nu}_{x_{0},t}}d\,\sigma_{y,s}-
       \frac{1}{2}\vf(x_0,t)\, .
\end{gather*}

As far as the invertibility of \eqref{eq:coperator2} is concerned, this can be proved 
as in \cite[Th. 1.4, p.185]{FabesRivHeat}.
\end{proof}

\begin{teor}\label{th:t->0}
    Let $\vf\in L^{p}(\Sigma_2)$ ($1\leq p<\infty$). Then
    \begin{equation}\label{eq:limt->0}
\lim_{t\to 0^{+}} \int_{\Sigma_2}\vf(y,0)\, G(x-y,t)\, dy = \vf(x,0)
\end{equation}
for almost every $(x,0) \in \Sigma_2$.
\end{teor}
\begin{proof}
Putting $u\equiv 1$ in formula \eqref{eq:StokesH}, we find
$$
\int_{\Sigma_3}\frac{\partial }{\partial\overline{\nu}_{y,s}}G(x-y,t-s)\, d\sigma_{y,s} 
+\int_{\Sigma_2}G(x-y,t)\, dy =1
$$
for any $(x,t)\in\Omega_{T}$. 
Letting $t\to 0^{+}$, we obtain \eqref{eq:limt->0}  when $\vf\equiv 1$. 
Then the thesis will follow if we show that
$$
\lim_{t\to 0^{+}} \int_{\Sigma_2}[\vf(y,0)-\vf(x,0)])\, G(x-y,t)\, dy=0
$$
for almost every $(x,0) \in \Sigma_2$.  This can be proved in Lebesgue points by
standard arguments (see, e.g., \cite[p.91--92]{AmerioCalore}).
\end{proof}

The same results hold for the potentials related to the operator $H^*$.

\begin{teor}
\label{teor:salto1}
Let $\vf\in L^{p}(\Sigma_3)$ ($1\leq p<\infty$). Then 
\begin{gather}
\begin{gathered}
 \lim_{(y,s)\to (y_0,s)^{\pm}}\int_{\Sigma_3}\vf(x,t)\,\frac{\partial G(x-y,t-s)}{\partial\overline{\nu}_{x,t}}\,d\sigma_{x,t}  \\
                            =\pm\frac{1}{2}\vf(y_0,s)    +\int_{\Sigma_3}\vf(x,t)\,\frac{\partial G(x-y_{0},t-s)}{\partial\overline{\nu}_{x,t}}\,d\sigma_{x,t} 
\end{gathered}   \label{eq:formulasalto1}\\
\begin{gathered}
 \lim_{(y,s)\to (y_0,s)^{\pm}}\int_{\Sigma_3}\vf(x,t)\,\frac{\partial G(x-y,t-s)}{\partial\overline{\nu}_{y_{0},s}}\,d\sigma_{x,t} \\
                    =\mp\frac{1}{2}\vf(y_0,s)     +\int_{\Sigma_3}\vf(x,t)\,\frac{\partial G(x-y_{0},t-s)}{\partial\overline{\nu}_{y_{0},s}}\,d\sigma_{x,t} 
\end{gathered}   \label{eq:formulasalto2}
\end{gather}
for almost every $(y_0,s) \in \Sigma_3$. As in Theorem \ref{teor:salto_doppio}, the limits $(y,s)\to (y_0,s)^{+}$ ($(y,s)\to (y_0,s)^{-}$) have to be understood as 
internal (external) angular boundary values and the integrals on the right hand side exist as singular integrals. Moreover, for any $1<p<\infty$ and for any real $c\neq 0$, the operators
\begin{gather*}
c\, \vf(y,s) + \int_{\Sigma_3}\vf(x,t)\,\frac{\partial G(x-y,t-s)}{\partial\overline{\nu}_{x,t}}\,d\sigma_{x,t}\, ,\\
c\, \vf(y,s) + \int_{\Sigma_3}\vf(x,t)\,\frac{\partial G(x-y,t-s)}{\partial\overline{\nu}_{y,s}}\,d\sigma_{x,t} 
\end{gather*}
are invertible on $L^{p}(\Sigma_{3})$.
\end{teor}
\begin{proof}
We start by observing that, reasoning as in the proof of Theorem \ref{teor:salto_doppio},  we have
\begin{gather*}
\int_{\Sigma_{3}^{\infty}}\,\frac{\partial G(x-y,t-s)}{\partial\overline{\nu}_{x,t}}\,d\sigma_{x,t}=
\int_{\partial\Omega}d\sigma_{x} \int_{s}^{+\infty} \frac{\partial G(x-y,t-s)}{\partial\overline{\nu}_{x,t}}\, dt
\\
=  - \int_{\partial\Omega} \frac{\partial}{\partial\overline{\nu}_{x}} s(x,y)\, d\sigma_x\, .
\end{gather*}
Then we can repeat that proof to obtain \eqref{eq:formulasalto1}. In a similar way we get
\eqref{eq:formulasalto2} as in Theorem \ref{teor:salto_conorm}.
\end{proof}

\begin{teor}
    Let $\vf\in L^{p}(\Sigma_1)$ ($1\leq p<\infty$). Then
   $$
\lim_{s\to T^{-}}\int_{\Sigma_1}\vf(x,T)\, G(x-y,T-s)\, dx = \vf(y,T)
$$
for almost every $(y,T) \in \Sigma_1$.
\end{teor}
\begin{proof}
The proof is the same as in Theorem \ref{th:t->0}.
\end{proof}

\section{The classes $\mathscr{A}^{p}$ and  $\mathscr{A}^{p}_{*}$}\label{sec:AeA*}

Let $p$ be a real number such that $1\leq p<\infty$. 
We give the following

\begin{dfz}
Let $u\in L^{p}(\Omega_T)$. We say that $u$ belongs to the class $\mathscr{A}^{p}(\Omega_T)$ 
 if there exist two functions $A\in L^{p}(\Sigma)$ and $B\in L^{p}(\Sigma_3)$ such that 
\begin{equation}\label{eq:defAp}
\int_{\Omega_T}u\,  H^{*}v\,dxdt  =\int_{\Sigma_3}\left(B\, v -A\, \frac{\partial v}{\partial\overline{\nu}}\right)d\sigma
                             +\int_{\Sigma_1}A\, v\,d\sigma-\int_{\Sigma_2}A\, v\, d\,\sigma
\end{equation}
for all $v$ in $C^{\infty}(\mathbb{R}^{n+1})$.
\end{dfz}

Similarly we define
\begin{dfz}
Let $u\in L^{p}(\Omega_T)$. We say that $u$ belongs to the $\mathscr{A_{*}}^{p}(\Omega_T)$  if there exist two functions $A\in L^{p}(\Sigma)$ and $B\in L^{p}(\Sigma_3)$ such that 
\begin{equation}\label{eq:defA*p}
\int_{\Omega_T}u\,  Hv\,dxdt  =\int_{\Sigma_3}\left(B\, v -A\, \frac{\partial v}{\partial\overline{\nu}}\right)d\sigma
                             -\int_{\Sigma_1}A\, v\,d\sigma+\int_{\Sigma_2}A\, v\, d\,\sigma
\end{equation}
for all $v$ in $C^{\infty}(\mathbb{R}^{n+1})$.
\end{dfz}

\begin{oss}\label{rmk:smooth}
If $u\in \mathscr{A}^{p}(\Omega_T)$ ($u\in \mathscr{A_{*}}^{p}(\Omega_T)$)
then $u$ is a smooth solution in $\Omega_T$ of the equation $Hu=0$ ($H^{*}u=0$). This follows from the
hypoellipticity of the operator $H$  ($H^*$) (see, e.g., \cite[p.146]{FollandPDEs}, where
the arguments used for the heat operator can be easily adapted to $H$ and $H^*$).
\end{oss}

\begin{oss}\label{rmk:neighb}
In conditions \eqref{eq:defAp} and \eqref{eq:defA*p} it is sufficient to consider functions $v$ which are $C^\infty$ in a neighborhood of $\overline{\Omega}_T$.
\end{oss}

It will be useful to introduce other two classes of functions

\begin{dfz}
We say that $u$ belongs to the class $\mathscr{B}^{p}(\Omega_T)$ 
 if there exist two functions $A\in L^{p}(\Sigma)$ and $B\in L^{p}(\Sigma_3)$ such that 
\begin{equation}\label{eq:defBp}
\begin{gathered}
\int_{\Sigma_3}\left(A(y,s)\, \frac{\partial }{\partial\overline{\nu}_{y,s}}G(x-y,t-s) -B(y,s)\, G(x-y,t-s)\right)d\sigma_{y,s} \\
-\int_{\Sigma_1}A(y,T)\, G(x-y,t-T)\,dy +\int_{\Sigma_2}A(y,0)\, G(x-y,t)\, dy 
\\=
\begin{cases}
u(x,t)  & \text{if }  (x,t)\in\Omega_{T}, \\ 
0 & \text{if } (x,t) \notin \overline{\Omega}_{T}.
\end{cases}
\end{gathered}
\end{equation}
\end{dfz}

\begin{dfz}
 We say that $u$ belongs to the class $\mathscr{B}_{*}^{p}(\Omega_T)$ 
 if there exist two functions $A\in L^{p}(\Sigma)$ and $B\in L^{p}(\Sigma_3)$ such that 
\begin{gather*}
\int_{\Sigma_3}\left(A(x,t)\, \frac{\partial }{\partial\overline{\nu}_{x,t}}G(x-y,t-s) - B(x,t)\, G(x-y,t-s) \right)d\sigma_{x,t} \\
+\int_{\Sigma_1}A(x,T)\, G(x-y,T-s)\,dx -\int_{\Sigma_2}A(x,0)\, G(x-y,-s)\, dx
\\=
\begin{cases}
u(y,s)  & \text{if }  (y,s)\in\Omega_{T}, \\ 
0 & \text{if } (y,s) \notin \overline{\Omega}_{T}.
\end{cases}
\end{gather*}
\end{dfz}

Roughly speaking $\mathscr{A}^{p}(\Omega_T)$ ($\mathscr{A_{*}}^{p}(\Omega_T)$) 
is the class of solutions of the equation
$Hu=0$ ($H^{*}u=0$) such that $u$ and $\partial u/\partial \overline{\nu}$ do exist on the boundary in some
weak sense, while $\mathscr{B}^{p}(\Omega_T)$ ($\mathscr{B_{*}}^{p}(\Omega_T)$)  is the class of solutions of the same equation for which a representation Stokes formula holds. 

Actually these two spaces coincide.

\begin{teor}\label{th:A=B}
    $\mathscr{A}^{p}(\Omega_T)=\mathscr{B}^{p}(\Omega_T)$, 
    $\mathscr{A_{*}}^{p}(\Omega_T)=\mathscr{B_{*}}^{p}(\Omega_T)$.
\end{teor}
\begin{proof}
Suppose that $u$ belongs to $\mathscr{A}^{p}(\Omega_T)$. If $(x,t)\notin \overline{\Omega}_{T}$,
the function $G(x-\cdot,t-\cdot)$ is a  solution of the equation $Hu=0$ in $\Omega_{T}$ and $C^{\infty}$ in a  neighborhood of $\overline{\Omega}_{T}$.  Therefore,  formula \eqref{eq:defBp} for $(x,t)\notin \overline{\Omega}_{T}$ follows immediately from
\eqref{eq:defAp} and Remark \ref{rmk:neighb}.

Let now  $(x,t)$ be fixed in $\Omega_{T}$ and
consider a scalar function $\psi\in C^{\infty}(\mathbb{R}^{n+1})$ such that
$\psi(y,s)=0$ if $|y|\leq 1/2$, $|s|\leq 1/2$, and $\psi(y,s)=1$ if $|y|\geq 1$, $|s|\geq 1$.
Take $\eps>0$  such that $\overline{B_{\eps}(x)}\times [t-\eps,t+\eps]\subset \Omega_{T}$ and define
$$
v(y,s)= \psi\left(\frac{y-x}{\eps},  \frac{s-t}{\eps}\right) G(x-y,t-s)\, .
$$

Since $v$ is a $C^{\infty}$ function and coincide with $G(x-y,t-s)$ for $(y,s)$ in a neighborhood of $\partial\Omega_{T}$, by \eqref{eq:defAp} we get
\begin{gather*}
\int_{\Sigma_3}\left(A(y,s)\, \frac{\partial }{\partial\overline{\nu}_{y,s}}G(x-y,t-s) -B(y,s)\, G(x-y,t-s)\right)d\sigma_{y,s} \\
-\int_{\Sigma_1}A(y,T)\, G(x-y,t-T)\,dy +\int_{\Sigma_2}A(y,0)\, G(x-y,t)\, dy 
\\=
\int_{\Sigma_3}\left(A\, \frac{\partial v}{\partial\overline{\nu}} - B\, v \right)d\sigma
                             -\int_{\Sigma_1}A\, v\,d\sigma+\int_{\Sigma_2}A\, v\, d\,\sigma
\\
= - \int_{\Omega_T}u\,  H^{*}v\,dyds = -
 \int_{B_{\eps}(x)\times (t-\eps,t+\eps)}u\,  H^{*}v\,dyds \, .
\end{gather*}
The last equality holds because $H^{*}v(y,s)=H_{(y,s)}^{*}G(x-y,t-s)=0$ for
$|y-x|>\eps$, $|s-t|>\eps$. 

With obvious notations, by Green's formula \eqref{eq:Green}
we can write the last integral as
\begin{gather*}
-\int_{\Sigma^{\eps}_{3}}\left(v\frac{\partial u}{\partial\overline{\nu}}-u\frac{\partial v}{\partial\overline{\nu}}\right)d\sigma -  \int_{\Sigma^{\eps}_{1}}uv\,d\sigma +\int_{\Sigma^{\eps}_{2}}uv\,d\sigma \, .
\end{gather*}

Keeping in mind  \eqref{eq:StokesH}, this is equal to
\begin{gather*}
-\int_{\Sigma^{\eps}_{3}}\left(\frac{\partial u(y,s)}{\partial\overline{\nu}_{y,s}} \, G(x-y,t-s) -u(y,s)\, \frac{\partial }{\partial\overline{\nu}_{y,s}}G(x-y,t-s)\right)d\sigma_{y,s} \\
+\int_{\Sigma^{\eps}_{2}}u(y,t-\eps)\, G(x-y,\eps)\, dy 
=
u(x,t)
\end{gather*}
Note that we can apply Green's formula and \eqref{eq:StokesH}, because the function $u$ is $C^{\infty}$ in
a neighborhood of $B_{\eps}(x)\times (t-\eps,t+\eps)$ and it satisfies the equation $Hu=0$ (see Remark \ref{rmk:smooth}). This establishes formula \eqref{eq:defBp}.

Conversely, if $u\in \mathscr{B}^{p}(\Omega_T)$, known results of potential theory show that $u$ belongs to $L^{p}(\Omega_{T})$. Thanks to Remark \ref{rmk:neighb}, it is enough to prove that \eqref{eq:defAp} holds for any
$v\in \Ctond^{\infty}(A)$, $A$ being a domain such that $\overline{\Omega}_{T}\subset A$. From
\eqref{eq:defBp} we get
\begin{gather*}
 \int_{\Omega_T}u\,  H^{*}v\,dxdt   =  \int_{\Omega_T} H^{*}v(x,t) \\
\times \Bigg(
\int_{\Sigma_3}\left(A(y,s)\, \frac{\partial }{\partial\overline{\nu}_{y,s}}G(x-y,t-s) -B(y,s)\, G(x-y,t-s)\right)d\sigma_{y,s} \\
-\int_{\Sigma_1}A(y,T)\, G(x-y,t-T)\,dy +\int_{\Sigma_2}A(y,0)\, G(x-y,t)\, dy 
\Bigg) dxdt
\\= \int_{A} H^{*}v(x,t) \\
\times \Bigg(
\int_{\Sigma_3}\left(A(y,s)\, \frac{\partial }{\partial\overline{\nu}_{y,s}}G(x-y,t-s) -B(y,s)\, G(x-y,t-s)\right)d\sigma_{y,s} \\
-\int_{\Sigma_1}A(y,T)\, G(x-y,t-T)\,dy +\int_{\Sigma_2}A(y,0)\, G(x-y,t)\, dy 
\Bigg) dxdt
\end{gather*}
Applying Fubini's Theorem and keeping in mind \eqref{eq:suppcomp}, we get 
\eqref{eq:defAp}.

The proof that $\mathscr{A_{*}}^{p}(\Omega_T)=\mathscr{B_{*}}^{p}(\Omega_T)$ is similar.
\end{proof}

The next Theorems show that usual uniqueness results hold in the class 
$\mathscr{A}^{p}(\Omega_T)$ ($\mathscr{A_{*}}^{p}(\Omega_T)$)  for the Dirichlet problem
and for the mixed problem for the equation $Hu=0$ ($H^{*}u=0$).

\begin{teor}
   Let $1<p<\infty$. If $u$ is solution of the Dirichlet problem
   $$
   \begin{cases}
   u\in \mathscr{A}^{p}(\Omega_T) &\\
Hu=0  & \text{in }  \Omega_{T} \\ 
 A =0  & \text{a.e.\ on }  \Sigma_{2}\cup \Sigma_{3}\, , 
\end{cases}
   $$ 
   then $u=0$ in $\Omega_{T}$.
\end{teor}
\begin{proof}
Thanks to Theorem \ref{th:A=B}, $u$ belongs to $\mathscr{B}^{p}(\Omega_T)$ and then
\begin{gather*}
-\int_{\Sigma_3}B(y,s)\, G(x-y,t-s) d\sigma_{y,s} 
-\int_{\Sigma_1}A(y,T)\, G(x-y,t-T)\,dy  
\\=
\begin{cases}
u(x,t)  & \text{if }  (x,t)\in\Omega_{T}, \\ 
0 & \text{if } (x,t) \notin \overline{\Omega}_{T}.
\end{cases}
\end{gather*}

In particular
\begin{equation}\label{eq:u=iB}
u(x,t) =- \int_{\Sigma_3}B(y,s)\, G(x-y,t-s) d\sigma_{y,s} \, , \quad
(x,t)\in \Omega_{T}\, .
\end{equation}

We have also
$$
 \int_{\Sigma_3}B(y,s)\, G(x-y,t-s) d\sigma_{y,s} = 0 \, , \quad
 (x,t)\notin \overline{\Omega}_{T},\ t<T.
$$
Therefore, for any fixed $(x_{0},t)\in \Sigma_3$, we get
$$
 \int_{\Sigma_3}B(y,s)\, \frac{\partial G(x-y,t-s)}{\partial\overline{\nu}_{x_{0},t}}
 \, d\sigma_{y,s}=0
$$
for any $x\notin \overline{\Omega}$. By the jump formula \eqref{eq:formulasalto_conorm}, we obtain
$$
\frac{1}{2}B(x_0,t)
                                                         +\int_{\Sigma_3}B(y,s)\,\frac{\partial G(x_0-y,t-s)}{\partial\overline{\nu}_{x_{0},t}}\,d\sigma_{y,s} =0
$$
for almost any $(x_{0},t)\in \Sigma_3$. The invertibility of the operator on the left hand side 
(see Theorem \ref{teor:salto_doppio})
implies $B=0$ a.e.. 
The result follows from \eqref{eq:u=iB}.
\end{proof}

\begin{teor}\label{teor:unicitmixed}
   Let $1<p<\infty$. If $u$ is solution of the mixed  problem
   $$
   \begin{cases}
   u\in \mathscr{A}^{p}(\Omega_T) &\\
Hu=0  & \text{in }  \Omega_{T} \\ 
 A =0  & \text{a.e.\ on }  \Sigma_{2}\, , \\
 B=0  & \text{a.e.\ on }  \Sigma_{3}\, ,
\end{cases}
   $$ 
   then $u=0$ in $\Omega_{T}$.
\end{teor}
\begin{proof}
Theorem \ref{th:A=B} implies $u\in \mathscr{B}^{p}(\Omega_T)$ and then
\begin{gather*}
\int_{\Sigma_3}A(y,s)\, \frac{\partial }{\partial\overline{\nu}_{y,s}}G(x-y,t-s) \, d\sigma_{y,s} 
-\int_{\Sigma_1}A(y,T)\, G(x-y,t-T)\,dy  
\\=
\begin{cases}
u(x,t)  & \text{if }  (x,t)\in\Omega_{T}, \\ 
0 & \text{if } (x,t) \notin \overline{\Omega}_{T}.
\end{cases}
\end{gather*}

In particular
\begin{equation}\label{eq:u=dentro}
u(x,t)=\int_{\Sigma_3}A(y,s)\, \frac{\partial }{\partial\overline{\nu}_{y,s}}G(x-y,t-s) \, d\sigma_{y,s}
 \, , \quad
(x,t)\in \Omega_{T} 
\end{equation}
and
$$
\int_{\Sigma_3}A(y,s)\, \frac{\partial }{\partial\overline{\nu}_{y,s}}G(x-y,t-s) \, d\sigma_{y,s}
=0  \, , \quad
 (x,t)\notin \overline{\Omega}_{T},\ t<T.
$$
In view of \eqref{eq:formulasalto_doppio}, the last condition implies
$$
-\frac{1}{2}A(x,t) + \int_{\Sigma_3}A(y,s)\, \frac{\partial }{\partial\overline{\nu}_{y,s}}G(x-y,t-s) \, d\sigma_{y,s}
=0
$$
for almost any $(x,t)\in \Sigma_3$. The invertibility of the operator on the left hand side 
(see Theorem \ref{teor:salto_conorm})
implies $A=0$ a.e.. 
The result follows from \eqref{eq:u=dentro}.
\end{proof}

By similar proofs we find
\begin{teor}\label{teor:unicitfirstbvp}
   Let $1<p<\infty$. If $u$ is solution of the Dirichlet problem
   $$
   \begin{cases}
   u\in \mathscr{A}_{*}^{p}(\Omega_T) &\\
H^{*}u=0  & \text{in }  \Omega_{T} \\ 
 A =0  & \text{a.e.\ on }  \Sigma_{1}\cup \Sigma_{3}\, , 
\end{cases}
   $$ 
   then $u=0$ in $\Omega_{T}$.
\end{teor}

\begin{teor}
   Let $1<p<\infty$. If $u$ is solution of the mixed problem
   $$
   \begin{cases}
   u\in \mathscr{A}_{*}^{p}(\Omega_T) &\\
H^{*}u=0  & \text{in }  \Omega_{T} \\ 
 A =0  & \text{a.e.\ on }  \Sigma_{1}\, , \\
 B=0  & \text{a.e.\ on }  \Sigma_{3}\, ,
\end{cases}
   $$ 
   then $u=0$ in $\Omega_{T}$.
\end{teor}

The next result provides a representation formula for functions belonging to 
 $\mathscr{A}^{p}(\Omega_T)$.

\begin{teor}
Let $1<p<\infty$.
    The function $u$ belongs to $\mathscr{A}^{p}(\Omega_T)$ if and only if there exists
    $\vf\in L^{p}(\Sigma_{2}\cup\Sigma_{3})$ such that
$$
u(x,t)= \int_{\Sigma_{2}\cup\Sigma_{3}}\vf(y,s)\, G(x-y,t-s)\, d\sigma_{y,s}\, ,
\quad (x,t)\in\Omega_{T}\, .
$$
\end{teor}
\begin{proof}
\textit{Sufficiency.}  Suppose
\begin{equation}\label{eq:slpS3}
u(x,t)=\int_{\Sigma_{3}}\vf(y,s)\, G(x-y,t-s)\, d\sigma_{y,s}
\end{equation}
with $\vf\in L^{p}(\Sigma_{3})$. For any $v\in C^{\infty}(\mathbb{R}^{n+1})$
we have
\begin{gather*}
\int_{\Omega_T}u\,  H^{*}v\,dxdt  =
\int_{\Sigma_{3}}\vf(y,s)\, d\sigma_{y,s} \int_{\Omega_T}
G(x-y,t-s)\, H^{*}v(x,t)\, dxdt\, .
\end{gather*}
Keeping in mind Theorem \ref{teor:salto1}, representation formula
\eqref{eq:StokesH*} leads to
\begin{gather*}
\int_{\Omega_{T}} G(x-y,t-s) \, H^{*}v(x,t)\, dxdt = -\frac{1}{2}\, v(y,s)
\\
+\int_{\Sigma_3}\left(v(x,t)\, \frac{\partial }{\partial\overline{\nu}_{x,t}}G(x-y,t-s) - \frac{\partial v(x,t)}{\partial\overline{\nu}_{x,t}}\, G(x-y,t-s)\right)d\sigma_{x,t} \\
+\int_{\Sigma_1}v(x,T)\, G(x-y,T-s)\, dx 
\end{gather*}
for almost every $(y,s)\in \Sigma_{3}$. Then
\begin{gather*}
\int_{\Omega_T}u\,  H^{*}v\,dxdt  = \int_{\Sigma_{3}}\vf(y,s) \Big(  -\frac{1}{2}\, v(y,s)\\
+\int_{\Sigma_3}\left(v(x,t)\, \frac{\partial }{\partial\overline{\nu}_{x,t}}G(x-y,t-s) - \frac{\partial v(x,t)}{\partial\overline{\nu}_{x,t}}\, G(x-y,t-s)\right)d\sigma_{x,t} \\
+\int_{\Sigma_1}v(x,T)\, G(x-y,T-s)\, dx  \Big) d\sigma_{y,s} \\
= \int_{\Sigma_3}\left(B\, v -A\, \frac{\partial v}{\partial\overline{\nu}}\right)d\sigma
                             +\int_{\Sigma_1}A\, v\,d\sigma
\end{gather*}
where
\begin{gather*}
B(x,t)= -\frac{1}{2}\, \vf(x,t) + \int_{\Sigma_{3}}\vf(y,s)\,  \frac{\partial }{\partial\overline{\nu}_{x,t}}G(x-y,t-s)\, 
d\sigma_{y,s}\, ,\  (x,t)\in \Sigma_{3}\, \\
A(x,t)=  \int_{\Sigma_{3}} \vf(y,s)\, G(x-y,t-s)\, d\sigma_{y,s}\, ,\ (x,t)\in\ \Sigma_{1}\cup \Sigma_{3}\, .
\end{gather*}
This proves that the potential \eqref{eq:slpS3} belongs to $\mathscr{A}^{p}(\Omega_T)$.

Likewise, if
$$
u(x,t)= \int_{\Sigma_{2}}\vf(y,0)\, G(x-y,t)\, dy\, ,
$$
keeping in mind also Theorem \ref{th:t->0}, we find that \eqref{eq:defAp} holds with
\begin{gather*}
B(x,t)= \int_{\Sigma_{2}}\vf(y,0)\,  \frac{\partial }{\partial\overline{\nu}_{x,t}}G(x-y,t)\, 
dy\, ,  \ (x,t)\in \Sigma_{3}\,\\
A(x,t)=\int_{\Sigma_{2}}\vf(y,0)\,  G(x-y,t)\, 
dy\, , \ (x,t)\in\ \Sigma_{1}\cup \Sigma_{3}\, ,\\
A(x,0)= \vf(x,0)\, , \ (x,0)\in \Sigma_{2}\, .
\end{gather*}
This completes the proof of the sufficiency.

\textit{Necessity.} Let $u\in \mathscr{A}^{p}(\Omega_T)$.   Let us consider the simple layer potential
$$
v(x,t)= \int_{\Sigma_{2}\cup\Sigma_{3}}\vf(y,s)\, G(x-y,t-s)\, d\sigma_{y,s}
$$
where
\begin{equation}\label{eq:condA}
\vf(x,0)=A(x,0), \quad x\in\Omega
\end{equation}
and $\vf$ on $\Sigma_{3}$ is the solution of the integral equation
\begin{equation}\label{eq:condB}
\begin{gathered}
-\frac{1}{2}\vf(x,t) + \int_{\Sigma_{3}}\vf(y,s)\,  \frac{\partial }{\partial\overline{\nu}_{x,t}}G(x-y,t-s)\, d\sigma_{y,s}\\
= - \int_{\Sigma_{2}}A(y,0)\,  \frac{\partial }{\partial\overline{\nu}_{x,t}}G(x-y,t)\, dy + B(x,t)\, ,
\quad (x,t)\in \Sigma_{3}.
\end{gathered}
\end{equation}
Here $A$ and $B$ are the ones in \eqref{eq:defAp} (or equivalently, in \eqref{eq:defBp}).
Note that the integral equation \eqref{eq:condB} is solvable, because of the invertibility of
operator \eqref{eq:coperator2}.

By the Sufficiency part of the proof, we know that $v\in \mathscr{A}^{p}(\Omega_T)$. Then,
the function $w=u-v$ belongs to $\mathscr{A}^{p}(\Omega_T)$.  Conditions \eqref{eq:condA}
and \eqref{eq:condB} show that $w=0$ a.e.\ on   $\Sigma_{2}$ and 
 $\partial w/\partial\overline{\nu}=0$  a.e.\ on $\Sigma_{3}$. The result follows from
 the uniqueness theorem \ref{teor:unicitmixed}.
\end{proof}

By a similar proof, we get
\begin{teor}\label{th:slpinA*p}
Let $1<p<\infty$.
    The function $u$ belongs to $\mathscr{A}_{*}^{p}(\Omega_T)$ if and only if there exists
    $\vf\in L^{p}(\Sigma_{1}\cup\Sigma_{3})$ such that
$$
u(y,s)= \int_{\Sigma_{1}\cup\Sigma_{3}}\vf(x,t)\, G(x-y,t-s)\, d\sigma_{x,t}\, ,
\quad (y,s)\in\Omega_{T}\, .
$$
\end{teor}

\section{Parabolic polynomials}\label{sec:polynom}

We define the parabolic polynomials $v_{\alpha}(x,t)$ as the coefficient of $\xi^{\alpha}/\alpha!$ in the following power series expansion
$$
\exp[\langle x,\xi\rangle + t\, a_{hk}\xi_h \xi_k] = 
\sum_{|\alpha|=0}^{\infty} v_{\alpha}(x,t) \frac{\xi^{\alpha}}{\alpha!}\, .
$$
Here we adopt the usual conventions for a multiindex $\alpha=(\alpha_1,\ldots, \alpha_n)$
$$
|\alpha | = \alpha_1 +\cdots + \alpha_n, \ \alpha! = \alpha_1! \cdots \alpha_n!, \ \xi^{\alpha} =
\xi_{1}^{\alpha_1} \cdots \xi_{n}^{\alpha_n}\, .
$$

In other words
$$
v_{\alpha}(x,t) = D^{\alpha}_{\xi} \left( \exp[\langle x,\xi\rangle + t\, a_{hk}\xi_h \xi_k]\right) \Big|_{\xi=0}\, .
$$

The polynomial $v_{\alpha}(x,t)$ satisfies the equation $Hv_{\alpha}=0$ and it can be represented  as
\begin{equation}\label{eq:integrrepr}
v_{\alpha}(x,t) = \int_{\mathbb{R}^n}G(x-y,t)\, y^{\alpha} dy 
\end{equation}
for any $x\in \mathbb{R}^n$ and $t>0$. This is a consequence of the fact that, for $t>0$, we have
\begin{gather*}
\sum_{|\alpha|=0}^{\infty}  \frac{\xi^{\alpha}}{\alpha!} \int_{\mathbb{R}^n} G(x-y,t)\, y^{\alpha} dy =
\int_{\mathbb{R}^n} G(x-y,t) \sum_{|\alpha|=0}^{\infty}  \frac{y^{\alpha}\xi^{\alpha}}{\alpha!} dy \\
=
\int_{\mathbb{R}^n} G(x-y,t)  e^{\langle y,\xi \rangle}dy
= \exp[\langle x,\xi\rangle + t\, a_{hk}\xi_h \xi_k] \, .
\end{gather*}

From \eqref{eq:integrrepr} we  see that $v_{\alpha}$ is of degree $\alpha_j$ in $x_j$ and
\begin{equation}\label{eq:t=0}
v_{\alpha}(x,0)= x^{\alpha}, \qquad \forall\ x\in \mathbb{R}^n.
\end{equation}

The next lemma shows that the linear span of the system $\{v_{\alpha}\}$ 
is given by all the polynomials $p$ satisfying the equation $Hp=0$.

\begin{lem}\label{lem:1}
    Let $p(x,t)$ a polynomial such that $Hp=0$. There exist a non negative integer $m$ and real constants
    $c_{\alpha}$ ($|\alpha|\leq m$) such that
    \begin{equation}\label{eq:span}
p(x,t) = \sum_{|\alpha|\leq m} c_{\alpha} v_{\alpha}(x,t), \quad \forall\ (x,t)\in \mathbb{R}^{n+1}.
\end{equation}
\end{lem}
\begin{proof}
Since $p(x,0)$ is a polynomial of a certain degree $m$, there exist real constants
    $c_{\alpha}$ such that
    $$
    p(x,0)=  \sum_{|\alpha|\leq m} c_{\alpha} x^{\alpha}, \quad  \forall\ x\in \mathbb{R}^{n}.
    $$
    
    Let $q$ be the polynomial
    $$
    q(x,t)= p(x,t) -  \sum_{|\alpha|\leq m} c_{\alpha} v_{\alpha}(x,t)\, .
    $$
    
 We have $Hq=0$ and then
 $$
 q(x,t) = \int_{\mathbb{R}^n}G(x-y,t)\, q(y,0)\, dy , \quad x\in\mathbb{R}^n,\, t>0.
 $$
  On the other hand \eqref{eq:t=0} implies $q(y,0)=0$ for any $y\in \mathbb{R}^n$. Therefore
  $q(x,t)=0$ for any $x\in\mathbb{R}^n,\, t>0$, and \eqref{eq:span} follows
    by analyticity.    
\end{proof}

In a similar way we may define the polynomials
$$
w_{\alpha}(x,t) = D^{\alpha}_{\xi} \left( \exp[\langle x,\xi\rangle - t\, a_{hk}\xi_h \xi_k]\right) \Big|_{\xi=0}\, .
$$
They satisfy the equation $H^{*}w_{\alpha}=0$ and, as in Lemma \ref{lem:1}, we have
\begin{lem}
    Let $q(x,t)$ a polynomial such that $H^{*}q=0$. There exist a non negative integer $m$ and real constants
    $c_{\alpha}$ ($|\alpha|\leq m$) such that
    \begin{equation*}
q(x,t) = \sum_{|\alpha|\leq m} c_{\alpha} w_{\alpha}(x,t), \quad \forall\ (x,t)\in \mathbb{R}^{n+1}.
\end{equation*}
\end{lem}

\section{Completeness theorems}\label{sec:compl}

Before proving our main result we need a completeness result for a particular class of potentials.

Let  $B_{r}$ be the open ball defined by $|x|<r$.  Let us consider the
domain $B_{r}\times(-T^{*},T^{*})$, where $T^{*}>T$ and $r$ is such that 
$\overline{\Omega}\subset B_{r}$.  Denote by $\mathscr{S}$ the class of potentials
$$
\Phi(x,t)=\int_{\Sigma^{*}_{2}\cup \Sigma^{*}_{3}}\vf(y,s)G(x-y,t-s)\,d\sigma_{y,s}
$$
with $\vf$ varying in $C^{0}(\Sigma^{*}_{2}\cup \Sigma^{*}_{3})$, where $\Sigma^{*}_{2}=\{(x,-T^{*})\,\,|\,\,x\in B_{r}\}$, $\Sigma^{*}_{3}=\partial B_{r}\times 
(-T^{*},T^{*})$\, .

\begin{teor}
\label{th:completzausiliare}
Let $1\leq p<\infty$. 
The system 
$\left\{\Phi\big|_{\Sigma_{2}\cup\Sigma_{3}} \, \middle| \,  \Phi\in \mathscr{S}\right\}$
is complete in $L^{p}(\Sigma_{2}\cup\Sigma_{3})$.
\end{teor}
\begin{proof}
Thanks to the Hann-Banach theorem, we have to show that, if
$\beta\in L^{q}(\Sigma_2\cup\Sigma_3)$ ($q=p/(p-1)$) is such that
\begin{equation}\label{eq:condort}
\int_{\Sigma_2\cup\Sigma_3}\beta\, \Phi \,d\sigma=0,\quad\forall\, \Phiå\in\mathscr{S}\, ,
\end{equation}
then $\beta=0$ a.e.\ on $\Sigma_2\cup\Sigma_3$. 
Conditions \eqref{eq:condort} mean that
\begin{gather*}
\int_{\Sigma^{*}_{2}\cup \Sigma^{*}_{3}}\vf(y,s)d\sigma_{y,s}\int_{\Sigma_2\cup\Sigma_3}\beta(x,t)\, G(x-y,t-s)d\sigma_{x,t}=0
\end{gather*}
for any $\vf\in C^{0}(\Sigma^{*}_{2}\cup \Sigma^{*}_{3})$.  This implies that  the simple layer parabolic potential
\begin{equation*}
u(y,s)=\int_{\Sigma_2\cup\Sigma_3}\beta(x,t)\, G(x-y,t-s)\,d\sigma_{x,t}
\end{equation*}
vanishes on $\Sigma^{*}_{2}\cup \Sigma^{*}_{3}$.
We observe that $u$ satisfies the following exterior BVP 
$$
\begin{cases}
u\in C^{\infty}\left(\mathbb{R}^{n+1} \setminus \left[B_{r}\times(-T^{*},T^{*})\right]\right)\\
H^{*}u=0, & \text{in }\mathbb{R}^{n+1}\setminus  \left[B_{r}\times(-T^{*},T^{*})\right],\\
u=0, & \text{on } \Sigma^{*}_{2}\cup \Sigma^{*}_{3}.
\end{cases}
$$
The function $u$ vanishing at infinity, we have $u=0$ in $\mathbb{R}^{n+1}\setminus  \left[B_{r}\times(-T^{*},T^{*})\right]$.
This can be proved by means of very classical methods (see \cite[Sec. 2]{Picone29}).
We repeat these here for the sake of completeness.

Given $\eps>0$, we may find $R$ large enough so that $|u|<\eps$ on the boundary of $B_{R}\times(-R,R)$. 
Since $u=0$ on $\Sigma^{*}_{2}\cup \Sigma^{*}_{3}$, we have that $|u|<\eps$ on the
parabolic boundary (for the operator $H^{*}$)  of the set
$\left[B_{R}\times(-R,R)\right]\setminus  \left[B_{r}\times(-T^{*},T^{*})\right]$, which  is given by
$$
\widetilde{\Sigma}_{1}\cup \widetilde{\Sigma}_{3}\cup \Sigma^{*}_{2}\cup \Sigma^{*}_{3},
$$
where $\widetilde{\Sigma}_{1}=\{(x,R)\,\,|\,\,x\in B_{R}\}$, $ \widetilde{\Sigma}_{3}=\partial B_{R}\times 
(-R,R)$. By the maximum principle we get $|u|<\eps$ on $\left[B_{R}\times(-R,R)\right]\setminus  \left[B_{r}\times(-T^{*},T^{*})\right]$, from which easily follows $u=0$ in $\mathbb{R}^{n+1}\setminus \left[B_{r}\times(-T^{*},T^{*})\right]$.

Let $0\leq s\leq T$. Since $\mathbb{R}^{n}\setminus \overline{\Omega}$ is connected,
$u(\cdot,s)$ is there analytic and $u(y,s)=0$ for $|y|>r$ , we get $u(y,s)=0$ for any $y\notin \overline{\Omega}$, 
$0\leq s\leq T$. If $s\notin [0,T]$, for similar reasons, we find $u(y,s)=0$ for any $y\in\mathbb{R}^{n}$. In other words we have $u(y,s)=0$ for any $(y,s)\notin \overline{\Omega}_{T}$.

By the properties of the simple layer potential $u$ on $\Sigma_{3}$, we deduce
$u=0$ on $\Sigma_{3}$.  We have also
\begin{equation*}
u(y,T)=\int_{\Sigma_2\cup\Sigma_3}\beta(x,t)\, G(x-y,t-T)\,d\sigma_{x,t}=0
\end{equation*}
and then $u$ is solution of the following Dirichlet boundary value problem 
\begin{equation*}
\begin{cases}
H^{*}u=0\quad\text{in}\quad\Omega_T\\
u=0\quad\text{on}\quad\Sigma_1\cup\Sigma_3\, .
\end{cases}
\end{equation*} 

Since
$$
u(y,s)=\int_{\Sigma_{3}}\beta(x,t)\, G(x-y,t-s)\,d\sigma_{x,t}\, ,\quad
(y,s)\in \Omega_{T},
$$
Theorem \ref{th:slpinA*p}  shows that  $u$ belongs to $\mathscr{A_{*}}^{p}(\Omega_T)$. Then,
by Theorem \ref{teor:unicitfirstbvp}, we get $u= 0$ in $\Omega_{T}$. 
Jump formulas \eqref{eq:formulasalto2} imply
$$
\beta= \left( \frac{\partial u}{\partial \overline{\nu}}\right)^{-} - \left( \frac{\partial u}{\partial \overline{\nu}}\right)^{+} = 0
$$
a.e.\ on $\Sigma_{3}$. Therefore we can write
$$
u(y,s)=\int_{\Sigma_2}\beta(x,0)\, G(x-y,-s)\,d x =0\, , \quad
(y,s)\notin \Sigma_2\, .
$$
By Theorem \ref{th:t->0}, we deduce
$$
\beta(y,0) = \lim_{s\to 0^{-}}\int_{\Sigma_2}\beta(x,0)\, G(x-y,-s)\,d x =0
$$
a.e.\ on $\Sigma_{2}$ and this  completes the proof for $1<p<\infty$.
This implies also the completeness for $p=1$.
\end{proof}

\begin{teor}
Let $1\leq p<\infty$. The system of parabolic polynomials $\{v_{\alpha}\}$ introduced in
Section \ref{sec:polynom} is complete in $L^{p}(\Sigma_{2}\cup\Sigma_{3})$.
\end{teor}
\begin{proof}

Let $1<p<\infty$ and take $f\in L^{p}(\Sigma_{2}\cup\Sigma_{3})$. By Theorem \ref{th:completzausiliare}, given $\eps>0$, there exists $\Phi\in \mathscr{S}$ such that
$$
\norma{f-\Phi}_{L^{p}(\Sigma_{2}\cup\Sigma_{3})}<\eps\, .
$$

The cylinder $B_{r}\times(-T^{*},T^{*})$ being convex and then $H$-convex (see 
\cite[Cor.\ 3, p.351]{treves}, by a result of Malgrange \cite{malgrange}, we can find a sequence 
$\{\omega_{k}\}$ of polynomial
solutions of the equation $Hu=0$  such that $\omega_{k} \to \Phi$ in $C^{\infty}(B_{r}\times(-T^{*},T^{*}))$
(see \cite[Th. 4.4, p.282, and Remark 4.1, p.284]{treves}). This implies that we can find
a polynomial $\omega$ solution  of the equation $Hu=0$  such that
$$
\norma{\Phi-\omega}_{L^{p}(\Sigma_{2}\cup\Sigma_{3})}<\eps\, 
$$
and then
$$
\norma{f-\omega}_{L^{p}(\Sigma_{2}\cup\Sigma_{3})}<2\, \eps\, .
$$

Since $\omega$ can be written as a finite linear combinations of $v_{\alpha}$
(Lemma\ref{lem:1}), we have the completeness of the system $\{v_{\alpha}\}$
in $L^{p}(\Sigma_{2}\cup\Sigma_{3})$ ($1<p<\infty$).
Completeness when $p=1$ follows easily.
\end{proof}

Defining  $\mathscr{S^{*}}$ the class of potentials
$$
\Psi(y,s)=\int_{\Sigma^{*}_{1}\cup \Sigma^{*}_{3}}\psi(x,t)G(x-y,t-s)\,d\sigma_{x,t}
$$
with $\psi$ varying in $C^{0}(\Sigma^{*}_{1}\cup \Sigma^{*}_{3})$, where $\Sigma^{*}_{1}=\{(x,T^{*})\,\,|\,\,x\in B_{r}\}$, we have also
\begin{teor}
Let $1< p<\infty$. 
The system $\left\{\Psi\big|_{\Sigma_{1}\cup\Sigma_{3}}\, |\, \Psi\in \mathscr{S^{*}}\right\}$
is complete in $L^{p}(\Sigma_{1}\cup\Sigma_{3})$.
\end{teor}

The proof is similar to that of Theorem \ref{th:completzausiliare} and we omit the details. As a consequence we find
\begin{teor}
Let $1\leq p<\infty$. The system of parabolic polynomials $\{w_{\alpha}\}$ introduced in
Section \ref{sec:polynom} is complete in $L^{p}(\Sigma_{1}\cup\Sigma_{3})$.
\end{teor}

 \section*{Acknowledgement.}
 A. Cialdea is member of Gruppo Nazionale per l'Analisi Matematica, la Pro\-ba\-bi\-li\-t\`a e le loro Applicazioni (GNAMPA) of the Istituto Nazionale di Alta Matematica (INdAM)
 and acknowledges the support from the project ``Perturbation problems and asymptotics for elliptic differential equations: variational and potential theoretic methods'' funded by the European Union - Next Generation EU and by MUR ``Progetti di Ricerca di Rilevante Interesse Nazionale'' (PRIN) Bando 2022 grant 2022SENJZ3.

\end{document}